\documentclass[a4paper,11pt,twoside]{amsart}
\usepackage{amssymb}
\newcommand{\R}{\mathbb{R}}

\newcommand{\liealg}[1]{\mathfrak{#1}}
\newcommand{\skal}[2]{\langle{#1},{#2}\rangle}

\newtheorem{theorem}{Theorem}[section]
\newtheorem{lemma}[theorem]{Lemma}
\newtheorem{proposition}[theorem]{Proposition}
\newtheorem{corollary}[theorem]{Corollary}

\theoremstyle{definition}
\newtheorem{definition}[theorem]{Definition}

\theoremstyle{remark}
\newtheorem{remark}[theorem]{Remark}

\numberwithin{equation}{section}

\begin{document}
%
\title[Normal holonomy of spacelike submanifolds]{On the normal holonomy representation of spacelike submanifolds in pseudo-Riemannian space forms}
%
%
\author{Kordian L\"{a}rz}
\address{Kordian L\"{a}rz,\newline
	\indent Humboldt-Universit\"{a}t Berlin,\newline
	\indent Institut f\"{u}r Mathematik,\newline
	\indent Rudower Chaussee 25,\newline
	\indent D-12489 Berlin.}
\email{laerz@math.hu-berlin.de}
%
%
%
%
%
\begin{abstract}
	In this paper we study weakly irreducible holonomy representations of the normal connection of a spacelike submanifold in
	a pseudo-Riemannian space from. We associate screen representations to weakly irreducible normal holonomy groups and classify the screen representations having the Borel-Lichn\'erowicz property.
	In particular, we derive a classification of Lorentzian normal holonomy representations.
\end{abstract}
\maketitle
%
\section{Introduction}
Any spacelike submanifold $M \subset \tilde{M}$ in a pseudo-Riemannian space form\footnote{In fact, throughout this paper we may substitute space forms
													by arbitrary pseudo-Riemannian spaces of constant curvature}
$\tilde{M}$ induces an orthogonal decomposition $T\tilde{M}|_{M} = TM \oplus NM$. The projections define the induced connection $\nabla^{\perp}$ on the normal bundle $NM$ by
\begin{equation*}
		\nabla^{\tilde{M}}_{X}\xi = -A_{\xi}X + \nabla^{\perp}_{X}\xi
							\quad \text{for }\xi \in \Gamma(NM),~X \in \Gamma(TM),
\end{equation*}
where $A_{\xi}X:=-pr_{TM}(\nabla^{\tilde{M}}_{X}\xi)$ is the shape operator of $M \subset \tilde{M}$.
Using $\nabla^{\perp}$ we derive a (restricted) holonomy representation which we call the normal holonomy representation and whose Lie algebra is denoted by
$\liealg{hol}^{\perp} \subset \liealg{so}(N_{p}M)$.
In \cite{MR1023346} irreducible normal holonomy representations of submanifolds in Euclidean space forms have been shown to act as the isotropy representation of a semisimple Riemannian symmetric space.
\par
$\liealg{hol}^{\perp}$ is said to act weakly irreducible if all invariant proper subspaces are degenerate. If $A \subset N_{p}M$ is
invariant then $A^{\perp}$ is invariant. Hence, $\Xi :=A \cap A^{\perp}$ is an invariant, isotropic subspace of $N_{p}M$. Moreover, we conclude
$\liealg{hol}^{\perp} \subset \mbox{Stab}_{\liealg{so}(N_{p}M)}(\Xi)$. In the following we consider normal holonomy representations where $\Xi$ has
maximal dimension. We identify $N_{p}M$ with $\R^{p,q+p}$ by choosing a pseudo-orthonormal basis
$(v_{1},\ldots,v_{p},e_{1},\ldots,e_{q},w_{1},\ldots,w_{p})$ of $N_{p}M$ where $\Xi = \text{span}(v_{1},\ldots,v_{p})$, $\Xi^{\perp} = \text{span}(v_{1},\ldots,v_{p},e_{1},\ldots,e_{q})$ and
\begin{equation*}
	\skal{v_{i}}{v_{j}} = \skal{w_{1}}{w_{p}} = \skal{v_{i}}{e_{j}} = \skal{w_{i}}{e_{j}} = 0, \quad \skal{e_{i}}{e_{j}} = \skal{v_{i}}{w_{j}} =\delta_{ij}.
\end{equation*}
With respect to this basis we derive the identification
\begin{equation*}
	\mbox{Stab}_{\liealg{so}(N_{p}M)}(\Xi) = \left\{ \begin{pmatrix}
									 & -X_{1}^{T} & \\
									A & \vdots & * \\
									 & -X_{p}^{T} & \\
									0 & B & X_{1} \cdots X_{p}\\
									0 & 0 & -A^{T}
								\end{pmatrix}: \begin{matrix}
											A \in \liealg{gl}(p),\\
											B \in \liealg{so}(q),\\
											X_{j} \in \R^{q}
										\end{matrix}\right\}.
\end{equation*}
The projection $pr_{\liealg{so}(q)}:\mbox{Stab}_{\liealg{so}(N_{p}M)}\rightarrow \liealg{so}(q)$ defines the normal screen holonomy algebra $\liealg{g}:=pr_{\liealg{so}(q)}(\liealg{hol}^{\perp})$.
In this paper we classify the representations of $\liealg{g}$ having the Borel-Lichn\'erowicz property (see below). In particular, if $\liealg{g}$ acts irreducibly we show the connected Lie subgroup
$G \subset SO(q)$ with $Lie(G)=\liealg{g}$ to act as the isotropy representation of a semisimple Riemannian symmetric space. As an application we derive a classification of all normal holonomy
representations of spacelike submanifolds in Lorentzian space forms.
\section{On the classification of normal screen holonomy representations}
Define the curvature tensor $R^{\perp}:\Gamma(TM) \times \Gamma(TM) \times \Gamma(NM) \rightarrow \Gamma(NM)$ of $\nabla^{\perp}$ by
$R^{\perp}(X,Y):=[\nabla_{X}^{\perp},\nabla_{Y}^{\perp}]-\nabla_{[X,Y]}^{\perp}$. The key observation is that associated to $R^{\perp}$ there is an algebraic curvature tensor $\mathcal{R}$ on $NM$
generating the same endomorphism of $NM$. This idea has already been used in \cite{MR1023346}. Its generalization to the pseudo-Riemannian case is straightforward since the Ricci equation
\begin{equation*}
	\skal{R^{\perp}(X_{1},X_{2})\xi_{1}}{\xi_{2}} = \skal{[A_{\xi_{1}},A_{\xi_{2}}]X_{1}}{X_{2}}
\end{equation*}
and the self-adjointness of $A_{\xi}$ are still obeyed. Hence, we state it without proof:
\begin{lemma}\label{olmos-curvature}\par \noindent
	\begin{enumerate}
		\item \noindent
		Let $(e_{1},\ldots,e_{\dim M})$ be an orthonormal basis of $T_{p}M$. Then
		\begin{equation*}
			\mathcal{R}_{p}(\xi_{1},\xi_{2})\xi_{3} := \sum_{i=1}^{\dim M}{R^{\perp}_{p}(A_{\xi_{1}}e_{i},A_{\xi_{2}}e_{i})\xi_{3}}
		\end{equation*}
		is an algebraic curvature tensor on $N_{p}M$, i.e.,
		\begin{align*}
			\mathcal{R}(\xi_{1},\xi_{2}) &=-\mathcal{R}(\xi_{2},\xi_{1}),\\
			\skal{\mathcal{R}(\xi_{1},\xi_{2})\xi_{3}}{\xi_{4}} &=-\skal{\xi_{3}}{\mathcal{R}(\xi_{1},\xi_{2})\xi_{4}},\\
			\skal{\mathcal{R}(\xi_{1},\xi_{2})\xi_{3}}{\xi_{4}} &= \skal{\mathcal{R}(\xi_{3},\xi_{4})\xi_{1}}{\xi_{2}},\\
			\mathcal{R}(\xi_{1},\xi_{2})\xi_{3} &+ \mathcal{R}(\xi_{2},\xi_{3})\xi_{1} +\mathcal{R}(\xi_{3},\xi_{1})\xi_{2}=0.
		\end{align*}
		Moreover, $\skal{\mathcal{R}(\xi_{1},\xi_{2})\xi_{3}}{\xi_{4}} =
					-\frac{1}{2}Tr([A_{\xi_{1}},A_{\xi_{2}}] \circ [A_{\xi_{3}},A_{\xi_{4}}])$.
		\item \noindent
		$\mbox{span}\{\mathcal{R}_{p}(\xi_{1},\xi_{2}): \xi_{1},\xi_{2} \in N_{p}M \}
				= \mbox{span}\{ R^{\perp}_{p}(X,Y): X,Y \in T_{p}M \}$.
	\end{enumerate}
\end{lemma}
Let $\mathcal{R}^{\tau_{\gamma}}(\xi_{1},\xi_{2}):=
	\tau_{\gamma}^{-1} \circ \mathcal{R}_{\gamma(1)}(\tau_{\gamma}(\xi_{1}),\tau_{\gamma}(\xi_{2}))
											\circ \tau_{\gamma}$, where
$\tau_{\gamma}$ denotes the parallel displacement with the normal connection $\nabla^{\perp}$ along the piecewise smooth
curve $\gamma:[0,1] \rightarrow M$ with $\gamma(0)=p$.
Using the Ambrose-Singer theorem we conclude
\begin{equation*}
	\liealg{hol}^{\perp}_{p} = \mbox{span}\{\mathcal{R}^{\tau_{\gamma}}(\xi_{1},\xi_{2}): \xi_{1},\xi_{2} \in N_{p}M,~\gamma(0)=p\}.
\end{equation*}
Using the basis $(v_{1},\ldots,v_{p},e_{1},\ldots,e_{q},w_{1},\ldots,w_{p})$ of $N_{p}M$ and the identification from the introduction we observe
\begin{equation*}
	\liealg{g} = \mbox{span}\{pr_{\mathcal{E}} \circ \mathcal{R}^{\tau_{\gamma}}(\xi_{1},\xi_{2})|_{\mathcal{E}}: \xi_{1},\xi_{2} \in N_{p}M \},
\end{equation*}
where $\mathcal{E}:=\mbox{span}\{e_{1},\ldots,e_{q}\}$. For $1 \leq i,j \leq p$ we define
\begin{align*}
	\mathcal{P}_{0} &:= pr_{\mathcal{E}} \circ \mathcal{R}^{\tau_{\gamma}}|_{\mathcal{E} \times \mathcal{E} \times \mathcal{E}} \in \Lambda^{2}\mathcal{E}^{*} \otimes \liealg{g},\\
	\mathcal{P}_{i} &:= pr_{\mathcal{E}} \circ \mathcal{R}^{\tau_{\gamma}}(w_{i},\cdot)|_{\mathcal{E} \times \mathcal{E}} \in \mathcal{E}^{*} \otimes \liealg{g},\\
	\mathcal{Q}_{ij} &:= pr_{\mathcal{E}} \circ \mathcal{R}^{\tau_{\gamma}}(w_{i},w_{j})|_{\mathcal{E}} \in \liealg{g}.
\end{align*}
\begin{definition}
	Let $\liealg{h} \subset \liealg{so}(\mathcal{E},h)$ for some Euclidean vector space $(\mathcal{E},h)$.
	\begin{enumerate}
		\item \noindent
		The space of algebraic curvature tensors with values in $\liealg{h}$ is given by
		\begin{equation*}
			\mathcal{K}(\liealg{h}) := \{R \in \Lambda^{2}\mathcal{E}^{*} \otimes \liealg{h}: R(x,y)z + R(y,z)x + R(z,x)y =0 \}.
		\end{equation*}
		\item \noindent
		The space of algebraic weak curvature tensors with values in $\liealg{h}$ is given by
		\begin{equation*}
			\mathcal{B}_{h}(\liealg{h}) := \{Q \in \mathcal{E}^{*} \otimes \liealg{h}: h(Q(x)y,z) + h(Q(y)z,x) + h(Q(z)x,y) =0\}.
		\end{equation*}
	\end{enumerate}
\end{definition}
We will need the following simple observation:
\begin{proposition}\label{curv-decomp}
	$\liealg{g} = \mbox{span}\{ \mathcal{P}_{0}(Y_{1},Y_{2}),\mathcal{P}_{k}(Y_{k}),\mathcal{Q}_{ij}: Y_{\cdot} \in \mathcal{E}\}$.
\end{proposition}
\begin{proof}
	Fix $\xi_{1},\xi_{2} \in N_{p}M$ and $E \in \mathcal{E}$. Then we have
	\begin{equation*}
		pr_{\mathcal{E}}(\mathcal{R}^{\tau_{\gamma}}(\xi_{1},\xi_{2})E)
			= \sum_{k=1}^{q}{\skal{\mathcal{R}^{\tau_{\gamma}}(\xi_{1},\xi_{2})E}{e_{k}}e_{k}}.
	\end{equation*}
	For $i \in \{1,2\}$ we may write $\xi_{i}= \alpha^{j}_{i}v_{j} + Y_{i} + \beta^{j}_{i}Z_{j}$ with $Y_{i} \in \mathcal{E}$. Then
	\begin{align*}
		\skal{\mathcal{R}^{\tau_{\gamma}}(\xi_{1},\xi_{2})E}{e_{k}} &= \skal{\mathcal{R}(\tau_{\gamma}(\xi_{1}),\tau_{\gamma}(\xi_{2}))\tau_{\gamma}(E)}{\tau_{\gamma}(e_{k})}\\
			&= \alpha^{j}_{1}\alpha^{i}_{2}\skal{\mathcal{R}(\tau_{\gamma}(v_{j}),\tau_{\gamma}(v_{i}))\tau_{\gamma}(E)}{\tau_{\gamma}(e_{k})}\\
			&\quad + \alpha^{j}_{1}\skal{\mathcal{R}(\tau_{\gamma}(v_{j}),\tau_{\gamma}(Y_{2}))\tau_{\gamma}(E)}{\tau_{\gamma}(e_{k})}\\
			&\quad + \alpha^{i}_{2}\skal{\mathcal{R}(\tau_{\gamma}(Y_{1}),\tau_{\gamma}(v_{i}))\tau_{\gamma}(E)}{\tau_{\gamma}(e_{k})}\\
			&\quad + \alpha^{j}_{1}\beta^{i}_{2}\skal{\mathcal{R}(\tau_{\gamma}(v_{j}),\tau_{\gamma}(w_{i}))\tau_{\gamma}(E)}{\tau_{\gamma}(e_{k})}\\
			&\quad + \beta^{j}_{1}\alpha^{i}_{2}\skal{\mathcal{R}(\tau_{\gamma}(w_{j}),\tau_{\gamma}(v_{i}))\tau_{\gamma}(E)}{\tau_{\gamma}(e_{k})}\\
			&\quad + \beta^{j}_{1}\skal{\mathcal{R}(\tau_{\gamma}(w_{j}),\tau_{\gamma}(Y_{2}))\tau_{\gamma}(E)}{\tau_{\gamma}(e_{k})}\\
			&\quad + \beta^{i}_{2}\skal{\mathcal{R}(\tau_{\gamma}(Y_{1}),\tau_{\gamma}(w_{i}))\tau_{\gamma}(E)}{\tau_{\gamma}(e_{k})}\\
			&\quad + \beta^{j}_{1}\beta^{i}_{2}\skal{\mathcal{R}(\tau_{\gamma}(w_{j}),\tau_{\gamma}(w_{i}))\tau_{\gamma}(E)}{\tau_{\gamma}(e_{k})}\\
			&\quad + \skal{\mathcal{R}(\tau_{\gamma}(Y_{1}),\tau_{\gamma}(Y_{2}))\tau_{\gamma}(E)}{\tau_{\gamma}(e_{k})}.
	\end{align*}
	Using Lemma \ref{olmos-curvature} we derive
	\begin{equation*}
		\skal{\mathcal{R}(\tau_{\gamma}(v_{j}),\tau_{\gamma}(v_{i}))\tau_{\gamma}(E)}{\tau_{\gamma}(e_{k})} =
			-\skal{\underbrace{\mathcal{R}(\tau_{\gamma}(E),\tau_{\gamma}(e_{k}))\tau_{\gamma}(v_{j})}_{\in \Xi}}{\tau_{\gamma}(v_{i})} =0
	\end{equation*}
	and $\skal{\mathcal{R}(\tau_{\gamma}(v_{j}),\tau_{\gamma}(Y_{2}))\tau_{\gamma}(E)}{\tau_{\gamma}(e_{k})}
					=\skal{\mathcal{R}(\tau_{\gamma}(Y_{1}),\tau_{\gamma}(v_{i}))\tau_{\gamma}(E)}{\tau_{\gamma}(e_{k})}=0$.
	Moreover, the Bianchi identity for $\mathcal{R}$ implies
	\begin{align*}
		\skal{\mathcal{R}(\tau_{\gamma}(v_{j}),\tau_{\gamma}(w_{i}))\tau_{\gamma}(E)}{\tau_{\gamma}(e_{k})}
			&= -\skal{\underbrace{\mathcal{R}(\tau_{\gamma}(w_{i}),\tau_{\gamma}(E))\tau_{\gamma}(v_{j})}_{\in \Xi}}{\tau_{\gamma}(e_{k})}\\
			&\quad-\skal{\mathcal{R}(\tau_{\gamma}(E),\tau_{\gamma}(v_{j}))\tau_{\gamma}(w_{i})}{\tau_{\gamma}(e_{k})}\\
			&=\skal{\underbrace{\mathcal{R}(\tau_{\gamma}(w_{i}),\tau_{\gamma}(e_{k}))\tau_{\gamma}(E)}_{\in \Xi^{\perp}}}{\tau_{\gamma}(v_{j})}
	\end{align*}
	and $\skal{\mathcal{R}(\tau_{\gamma}(w_{j}),\tau_{\gamma}(v_{i}))\tau_{\gamma}(E)}{\tau_{\gamma}(e_{k})}=0$. Therefore we conclude
	\begin{align*}
		\skal{\mathcal{R}^{\tau_{\gamma}}(\xi_{1},\xi_{2})E}{e_{k}} &= \skal{\mathcal{P}_{0}(Y_{1},Y_{2})(E)}{e_{k}}\\
										&\quad + \skal{\mathcal{P}_{j}(\beta^{j}_{1}Y_{2}-\beta^{j}_{2}Y_{1})(E)}{e_{k}}\\
										&\quad + \skal{\beta^{i}_{1}\beta^{j}_{2}\mathcal{Q}_{ij}(E)}{e_{k}}.
	\end{align*}
\end{proof}
\begin{remark}
	Using the definition and Lemma \ref{olmos-curvature} we derive $\mathcal{P}_{0} \in \mathcal{K}(\liealg{g})$ and $\mathcal{P}_{k} \in \mathcal{B}_{h|_{\mathcal{E}}}(\liealg{g})$.
	A Lie algebra $\liealg{h}$ with $\liealg{h}=\mbox{span}\{Q(x):~x \in \mathcal{E},~Q \in \mathcal{B}_{h}(\liealg{h})\}$ is called weak Berger algebra. The computations above
	imply that $\liealg{g}$ is a weak Berger algebra if the $Q_{ij}$ are generated by $\mathcal{P}_{0}$ and $\mathcal{P}_{k}$. This happens, e.g., if $NM$ has Lorentzian signature. Moreover,
	representations of weak Berger algebras have been classified in \cite{MR2331527}. There each weak Berger algebra is shown to act as the holonomy representation of a Riemannian manifold.
\end{remark}
\bigskip
We would like to restrict to irreducible normal screen holonomy representations. For submanifolds in Lorentzian space forms this approach is justified by
\begin{proposition}\label{screen-decomp}
	Let $\liealg{g} \subset \liealg{so}(\mathcal{E})$ be the normal screen holonomy algebra of a submanifold in a Lorentzian space form. Then there is an orthogonal decomposition
	$\mathcal{E}=E_{0} \oplus \ldots \oplus E_{\ell}$ such that	$E_{0}$ is a trivial submodule and $E_{j}$ are irreducible. Moreover, there is a decomposition
	$\liealg{g}=\liealg{g}_{1}\oplus \ldots \oplus \liealg{g}_{\ell}$ into commuting ideals such that $\liealg{g}_{j}$ acts irreducibly on $E_{j}$ and trivially on $E_{i}$ for $i \neq j$.
\end{proposition}
\begin{proof}
	We may use a similar approach as in \cite{MR1216527}. If $V \subset E$ is invariant then $E$ decomposes into $E=V \oplus V^{\perp}$ since $(E,h)$ is Euclidean. Let $g_{V}$ be the subalgebra of
	$\liealg{so}(\mathcal{E})$ which leaves $V$ invariant and annihilates $V^{\perp}$. Define $\liealg{g}_{1}:= \liealg{g} \cap g_{V}$ and $\liealg{g}_{2}:= \liealg{g} \cap g_{V^{\perp}}$.
	We need to show $\liealg{g} \subset g_{V} \oplus g_{V^{\perp}}$. Let $Y_{1},\tilde{Y}_{1} \in V$ and $Y_{2},\tilde{Y}_{2} \in V^{\perp}$. If the ambient space has Lorentzian signature we conclude
	$\mathcal{Q}_{ij}=0$ for all $i,j$. Moreover, using the algebraic curvature properties for $\mathcal{R}$ we observe
	\begin{equation*}
		\begin{matrix}
			\mathcal{P}_{0}(Y_{1},Y_{2})=0, & \quad & \mathcal{P}_{0}(Y_{1},\tilde{Y}_{1})|_{V^{\perp}}=0,& \quad & \mathcal{P}_{0}(Y_{2},\tilde{Y}_{2})|_{V}=0,\\
			\mathcal{P}_{k}(Y_{1})|_{V^{\perp}}=0, & \quad & \mathcal{P}_{k}(Y_{2})|_{V} =0.
		\end{matrix}
	\end{equation*}
	Hence, the proof follows using Proposition \ref{curv-decomp}.
\end{proof}
\begin{definition}
	Let $\liealg{g} \subset \liealg{so}(\mathcal{E})$. We say $\liealg{g}$ has the Borel-Lichn\'erowicz property if there are decompositions $\mathcal{E}=E_{0} \oplus \ldots \oplus E_{\ell}$
	and $\liealg{g}=\liealg{g}_{1}\oplus \ldots \oplus \liealg{g}_{\ell}$ such that each $\liealg{g}_{j}$ acts irreducibly on $E_{j}$ and trivially on $E_{i}$ for $i \neq j$.
\end{definition}
\begin{remark}
	According to Proposition \ref{screen-decomp} the normal screen holonomy algebra $\liealg{g}$ has the Borel-Lichn\'erowicz property if the ambient space is a Lorentzian space form.
	It is not known if this is true for ambient spaces with arbitrary signature. In the proof we have only used that the $\mathcal{Q}_{ij}$ are generated by $\mathcal{P}_{0}$
	and $\mathcal{P}_{k}$.
\end{remark}
\begin{lemma}\label{keylemma}
	Let $\liealg{g} \subset \liealg{so}(\mathcal{E})$ be the normal screen holonomy algebra of a submanifold in a pseudo-Riemannian space form. Assume $\liealg{g}$ has the Borel-Lichn\'erowicz property
	with $\liealg{g}=\liealg{g}_{1} \oplus \ldots \oplus \liealg{g}_{\ell}$. Then $\mathcal{K}(\liealg{g}_{j}) \neq 0$.
\end{lemma}
\begin{proof}
	If $\mathcal{K}(\liealg{g}_{j})=0$ we conclude
	\begin{equation*}
		\liealg{g}_{j} = \mbox{span}\{\mathcal{P}_{k}|_{E_{j}},\mathcal{Q}|_{E_{j}}:~k \neq 0 \}.
	\end{equation*}
	using Proposition \ref{curv-decomp}. Let $e_{1},e_{2},e_{3} \in E_{j}$ and write $\tau_{\gamma}e_{i} = \tilde{e}_{i} + V_{i}$ for $\tilde{e}_{i} \in E_{j} \subset \Xi^{\perp}$ and $V_{i} \in \Xi$.
	By Lemma \ref{olmos-curvature} we have
	\begin{align*}
		0 = \skal{\mathcal{P}_{0}(e_{2},e_{3})e_{2}}{e_{3}} &= \skal{\mathcal{R}^{\tau_{\gamma}}(e_{2},e_{3})e_{2}}{e_{3}}\\
			&= -\frac{1}{2}Tr([A_{\tau_{\gamma}(e_{2})},A_{\tau_{\gamma}(e_{3})}] \circ [A_{\tau_{\gamma}(e_{2})},A_{\tau_{\gamma}(e_{3})}]).
	\end{align*}
	Since $M$ is spacelike $Tr(A \circ B^{T})$ defines an inner product on skewsymmetric operators. Therefore, we derive $[A_{\tau_{\gamma}(e_{2})},A_{\tau_{\gamma}(e_{3})}]=0$.
	On the other hand
	\begin{align*}
		\skal{\mathcal{P}_{k}(e_{1})e_{2}}{e_{3}} &= \skal{\mathcal{R}^{\tau_{\gamma}}(w_{k},e_{1})e_{2}}{e_{3}}\\
					&= \skal{\mathcal{R}(\tau_{\gamma}(w_{k}),\tau_{\gamma}(e_{1}))\tau_{\gamma}(e_{2})}{\tilde{e}_{3}}
									+ \skal{\underbrace{\mathcal{R}(\tau_{\gamma}(w_{k}),\tau_{\gamma}(e_{1}))\tau_{\gamma}(e_{2})}_{\in \Xi^{\perp}}}{V_{3}}\\
					&= \skal{\mathcal{R}(\tau_{\gamma}(w_{k}),\tau_{\gamma}(e_{1}))\tilde{e}_{2}}{\tilde{e}_{3}}
									+\skal{\underbrace{\mathcal{R}(\tau_{\gamma}(w_{k}),\tau_{\gamma}(e_{1}))V_{2}}_{\in \Xi}}{\tilde{e}_{3}}\\
					&=-\frac{1}{2}Tr([A_{\tau_{\gamma}(w_{k})},A_{\tau_{\gamma}(e_{1})}] \circ [A_{\tilde{e}_{2}},A_{\tilde{e}_{3}}])
	\end{align*}
	and
	\begin{align*}
		\skal{\mathcal{Q}_{ij}e_{2}}{e_{3}} &= \skal{\mathcal{R}^{\tau_{\gamma}}(w_{i},w_{j})e_{2}}{e_{3}}\\
				&=-\frac{1}{2}Tr([A_{\tau_{\gamma}(w_{i})},A_{\tau_{\gamma}(w_{j})}] \circ [A_{\tilde{e}_{2}},A_{\tilde{e}_{3}}]).
	\end{align*}
	However, by the Ricci equation we have
	\begin{equation*}
		0 = \skal{R^{\perp}(X,Y)V_{i}}{\tilde{e}_{j}} = \skal{[A_{V_{i}},A_{\tilde{e}_{j}}]X}{Y},
	\end{equation*}
	i.e., $[A_{V_{i}},A_{\tilde{e}_{j}}]=0$. Hence, we conclude
	\begin{align*}
		\skal{\mathcal{P}_{k}(e_{1})e_{2}}{e_{3}} &= -\frac{1}{2}Tr([A_{\tau_{\gamma}(w_{k})},A_{\tau_{\gamma}(e_{1})}] \circ [A_{\tau_{\gamma}(e_{2})},A_{\tau_{\gamma}(e_{3})}])=0,\\
		\skal{\mathcal{Q}_{ij}e_{2}}{e_{3}} &= -\frac{1}{2}Tr([A_{\tau_{\gamma}(w_{i})},A_{\tau_{\gamma}(w_{j})}] \circ [A_{\tau_{\gamma}(e_{2})},A_{\tau_{\gamma}(e_{3})}])=0.
	\end{align*}
	Therefore $\liealg{g}_{j}=0$ and we have a contradiction.
\end{proof}
Observe that we did not use any irreducibility conditions in the proof of Lemma \ref{keylemma}. However, we need the normal screen holonomy to act irreducibly for the proof of our main result:
\begin{theorem}\label{mainscreenthm}
	Let $\liealg{g} \subset \liealg{so}(\mathcal{E})$ be the normal screen holonomy algebra of a submanifold in a pseudo-Riemannian space form. Assume $\liealg{g}$ has the Borel-Lichn\'erowicz property
	and let $G_{j} \subset SO(E_{j})$ be the connected Lie subgroup with $Lie(G_{j})=\liealg{g}_{j}$. Then $G_{j}$ acts irreducibly on $E_{j}$ as the isotropy representation of a semisimple
	Riemannian symmetric space if $j > 0$.
\end{theorem}
\begin{proof}
	Given the observations from Proposition \ref{curv-decomp} and Lemma \ref{keylemma} we can find
	\begin{equation*}
		R_{j}:= pr_{E_{j}} \circ \mathcal{R}^{\tau_{\gamma}}|_{E_{j} \times E_{j} \times E_{j}} \neq 0.
	\end{equation*}
	Moreover, $\liealg{g}_{j}$ and therefore $G_{j} \subset SO(E_{j})$ act irreducibly on $E_{j}$, i.e., $G_{j}$ is compact and $(E_{j},R_{j},G_{j})$ is an irreducible holonomy system in the
	sense of \cite{MR0148010}. Hence, the statement follows from \cite{MR0148010}[Thm. 5] once we have shown $\text{scal}(R_{j}) \neq 0$ where
	$\text{scal}(R_{j})=\sum_{k,\ell=1}^{\dim E_{j}}{\skal{R_{j}(e_{k},e_{\ell})e_{\ell}}{e_{k}}}$ is the scalar curvature of $R_{j}$. We compute
	\begin{align*}
		\sum_{k,\ell=1}^{\dim E_{j}}{\skal{R_{j}(e_{k},e_{\ell})e_{\ell}}{e_{k}}} &= \sum_{k,\ell=1}^{\dim E_{j}}{\skal{\mathcal{R}^{\tau_{\gamma}}(e_{k},e_{\ell})e_{\ell}}{e_{k}}}\\
			&= \sum_{k,\ell=1}^{\dim E_{j}}{\skal{\mathcal{R}(\tau_{\gamma}(e_{k}),\tau_{\gamma}(e_{\ell}))\tau_{\gamma}(e_{\ell})}{\tau_{\gamma}(e_{k})}}\\
			&= -\frac{1}{2}\sum_{k,\ell=1}^{\dim E_{j}}{Tr([A_{\tau_{\gamma}(e_{k})},A_{\tau_{\gamma}(e_{\ell})}] \circ [A_{\tau_{\gamma}(e_{\ell})},A_{\tau_{\gamma}(e_{k})}])}\\
			&=-\frac{1}{2}\sum_{k,\ell=1}^{\dim E_{j}}{\{ [A_{\tau_{\gamma}(e_{k})},A_{\tau_{\gamma}(e_{\ell})}],[A_{\tau_{\gamma}(e_{k})},A_{\tau_{\gamma}(e_{\ell})}]\} }\\
			&<0,
	\end{align*}
	where $\{A,B\} = Tr(A \circ B^{T})$ is the usual inner product on skewsymmetric operators. The last inequality follows since $\skal{R_{j}(e_{k},e_{\ell})e_{\ell}}{e_{k}} \neq 0$ for some $k,\ell$
	unless $R_{j}=0$.
\end{proof}
\section{Applications}
\subsection{Lorentzian normal holonomy representations}\hfill \par
As an application of the previous results we want to derive a classification of all normal holonomy representations of spacelike submanifolds $M \subset \tilde{M}$ in Lorentzian space forms.
Using the same approach as in \cite{MR1023346} we get a orthogonal holonomy invariant splitting
\begin{equation*}
	NM=NM_{0} \oplus NM_{L} \oplus NM_{1} \oplus \ldots \oplus NM_{r}
\end{equation*}
of the normal bundle and a splitting
\begin{equation*}
	Hol^{\perp}_{0}(M) = Hol^{\perp}_{L} \times Hol^{\perp}_{1} \times \ldots \times Hol^{\perp}_{r}
\end{equation*}
of the restricted normal holonomy group such that\footnote{$NM_{0}$ may have Lorentzian signature and $Hol^{\perp}_{L}$ may vanish. E.g., if $M \subset \mathbb{H}^{n} \subset \R^{1,n}$
									where $\mathbb{H}^{n}$ is the hyperbolic space then using the position vector field we
									get a timelike $\nabla^{\perp}$-parallel normal vector field on $M \subset \R^{1,n}$.}
\begin{itemize}
	\item \noindent
	$NM_{0}$ is a maximal flat and non-degenerate subspace and $NM_{j}$ has Riemannian signature for $j \geq 1$,
	\item \noindent
	$Hol^{\perp}_{j}$ acts irreducibly on $NM_{j}$ for $j \geq 1$ and trivially on $NM_{i}$ for $i \neq j$,
	\item \noindent
	$Hol^{\perp}_{L}$ acts weakly irreducibly on $NM_{L}$ and trivially on $NM_{i}$ for $i \neq L$.
\end{itemize}
The actions of $Hol^{\perp}_{j}$ have been classified in \cite{MR1023346}. Moreover, if $Hol^{\perp}_{L} \subset SO_{0}(1,m+1)$ acts irreducibly then by a well known result \cite{MR1836778} we conclude
$Hol^{\perp}_{L} = SO_{0}(1,m+1)$. Hence, we have to classify the weakly irreducible but non-irreducible representations. For such representations we have
\begin{theorem}[B{\'e}rard-Bergery \& Ikemakhen \cite{MR1216527}]\label{BBIthm}
	Let $\liealg{h} \subset \liealg{so}(1,m+1)$ be the Lie algebra of the connected Lie group $H \subset SO(1,m+1)$
	acting weakly irreducibly but non-irreducibly.
	Then $\liealg{h} \subset (\R \oplus \liealg{so}(m)) \ltimes \R^{m}$. Moreover, using
	$\liealg{g}:= pr_{\liealg{so}(m)}(\liealg{h})$ the Lie algebra $\liealg{h}$ belongs to one of the following
	types:
	\begin{itemize}
		\item \noindent
		Type 1: $\liealg{h} = (\R \oplus \liealg{g}) \ltimes \R^{m}$
		\item \noindent
		Type 2: $\liealg{h} = \liealg{g} \ltimes \R^{m}$
		\item \noindent
		Type 3:
		\begin{equation*}
			\liealg{h} = \left\lbrace \begin{pmatrix}
							\varphi(A) & w^T & 0\\
							0 & A & w\\
							0 & 0 & -\varphi(A)
						\end{pmatrix}: A \in \liealg{g},~w \in \R^{m} \right\rbrace
		\end{equation*}
		where $\varphi: \liealg{g} \twoheadrightarrow \R$ is an epimorphism satisfying
		$\varphi|_{[\liealg{g},\liealg{g}]}=0$.
		\item \noindent
		Type 4:
		There is $0< \ell < m$ such that $\R^{m} = \R^{\ell} \oplus \R^{m-\ell}$,
		$\liealg{g} \subset \liealg{so}(\ell)$ and
		\begin{equation*}
			\liealg{h} = \left\lbrace \begin{pmatrix}
							0 & \psi(A)^T & w^T & 0\\
							0 & 0 & 0 & -\psi(A)\\
							0 & 0 & A & -w\\
							0 & 0 & 0 & 0
						\end{pmatrix}: A \in \liealg{g},~w \in \R^{\ell}  \right\rbrace
		\end{equation*}
		for some epimorphism $\psi: \liealg{g} \twoheadrightarrow \R^{m-\ell}$ satisfying
		$\psi|_{[\liealg{g},\liealg{g}]}=0$.
	\end{itemize}
\end{theorem}
In our case the Lie algebra $\liealg{g}$ in the previous theorem is the normal screen holonomy algebra. Combining the Theorems \ref{BBIthm}, \ref{mainscreenthm} and Proposition \ref{screen-decomp}
we derive
\begin{corollary}
	Let $M$ be a spacelike submanifold in a Lorentzian space form. Then the weakly irreducible part of the normal holonomy representation is given by one of the types in Thm. \ref{BBIthm}
	and its normal screen holonomy representation acts as the isotropy representation of a Riemannian symmetric space.
\end{corollary}
\subsection{Submanifolds with $\nabla^{\perp}$-parallel second fundamental form}\hfill \par
Let $M \subset \R^{1,N}$ be a full spacelike submanifold and $\Pi(X,Y):=\nabla^{\R^{1,N}}_{X}Y -\nabla^{M}_{X}Y$ its
second fundamental form. Moreover, assume $\nabla^{\perp}\Pi =0$.\footnote{Full extrinsically symmetric submanifolds
								provide examples for this class of submanifolds.
								A detailed study can be found in \cite{kath-2008}.}
In order to study the possible normal holonomy groups we will apply ideas from \cite{MR0367876}. Let $\xi \in \Gamma(NM)$
and $X,Y,Z \in \Gamma(TM)$. Then
\begin{equation*}
	\skal{(\nabla^{\perp}_{X}\Pi)(Y,Z)}{\xi} = \skal{(\nabla_{X}A_{\xi})Y}{Z} -\skal{A_{\nabla^{\perp}_{X}\xi}Y}{Z}.
\end{equation*}
The mean curvature vector field $H:=\frac{1}{\dim M}\sum_{i=1}^{\dim M}{\Pi(e_{i},e_{i})}=\frac{Tr(\Pi)}{\dim M}$ is 
$\nabla^{\perp}$-parallel and therefore $\nabla_{X}A_{H}=0$. Using \cite{MR0367876}[Lemma 1] we conclude that $A_{H}$
has constant eigenvalues $(\lambda_{1},\ldots,\lambda_{r})$ and parallel eigendistributions. Since $M$ is full we may apply \cite{kath-2008}[Lemma 3.2] and Moore's Lemma. Hence, $M$
is locally a product of immersions
\begin{equation*}
	M=M_{1}\times \ldots \times M_{r} \rightarrow \R^{1,n} \times \R^{n_{2}} \ldots \R^{n_{r}}.
\end{equation*}
Consider the local normal holonomy of the full immersion $M_{1} \rightarrow \R^{1,n}$. By construction
$\nabla^{\perp}\Pi_{1}=0$ and $A_{H}=\lambda \cdot \text{id}_{TM_{1}}$. The following cases may occur:\par \noindent
Case 1: $H$ is timelike, i.e., the Lorentzian part of the local normal holonomy representation vanishes.\par \noindent
Case 2: $H \neq 0$ and $H \in NM_{L}$. Then $\|H\|=0$ and
\begin{equation*}
	\lambda \skal{X}{Y} = \skal{A_{H}X}{Y} = \skal{\Pi_{1}(X,Y)}{H},
\end{equation*}
i.e., $\lambda = \skal{H}{H}=0$ and therefore $A_{H}=0$. However, this would imply that $M_{1}$ is not full.\par \noindent
Case 3: $H \notin NM_{L}$ or $H=0$ and moreover $Hol^{\perp}_{L,loc}(M_{1}) \neq 0$ is of type 2 or of type 4. Then $\Xi$ is locally spanned by a $\nabla^{\perp}$-parallel lightlike vector
field $\xi \neq 0$. In the same way as for $H$ we conclude $\nabla^{\perp}(A_{\xi})=0$ and $A_{\xi}=\mu \cdot \text{id}_{TM_{1}}$. Since $\mu = \skal{\xi}{H}=0$ we
derive a contradiction.\par \noindent
Observe that the $Hol^{\perp}_{L}(M)$ is only of type 2 or of type 4 if $Hol^{\perp}_{L,loc}(M_{1})$ is somewhere. Hence, we conclude
\begin{proposition}
	Let $M \subset \R^{1,N}$ be a full spacelike submanifold with $\nabla^{\perp}\Pi=0$. Then the normal holonomy is 
	reducible and if $Hol^{\perp}_{L} \neq 0$ then it is of type 1 or 3.
\end{proposition}
I do not know if $M$ can have a normal holonomy representation of type 1 or 3.
\subsection{Submanifolds of the light cone}\hfill \par
For any spacelike submanifold $M \subset \R^{1,N}$ with $Hol^{\perp}_{L} \neq 0$ the distribution $\Xi$ is locally spanned
by a $\nabla^{\perp}$-parallel lightlike vector field $\xi$ if $Hol^{\perp}_{L}$ is of type 2 or 4. Now we assume that
$\xi$ is globally defined on $M$ and moreover $A_{\xi}=\lambda \cdot \text{id}_{TM}$ for some $\lambda \neq 0$.
Consider the position vector field $V: \R^{1,N} \rightarrow \R^{1,N}$ with $V(p)=p$. Then
$\nabla^{\R^{1,N}}_{\cdot}V = \text{id}_{\R^{1,N}}$ and therefore
\begin{equation*}
	\nabla^{\R^{1,N}}_{X}{(\xi + \lambda V)} =0 \quad \text{for } X \in TM.
\end{equation*}
Hence, we can find $V_{0} \in \R^{1,N}$ such that $V = V_{0} -\frac{1}{\lambda}\xi$, i.e., after a translation by $V_{0}$
we have $M \subset L^{N}$. Here $L^{N}$ is the light cone in $\R^{1,N}$. Conversely, let $M \subset L^{N}$ and $V$ be the
restriction of the position vector field to $M$. Then $V \in \Gamma(NM)$ and
\begin{equation*}
	\skal{V}{V}=0, \quad \nabla^{\perp}V=0, \quad A_{V} = -\text{id}_{TM}.
\end{equation*}
This follows from $\nabla^{\R^{1,N}}_{\cdot}V = \text{id}_{\R^{1,N}}$ and
$0 = X(\skal{V}{V}) = 2\skal{\nabla^{\R^{1,N}}_{X}V}{V}=2\skal{X}{V}$ for $X \in TM$.\par
Finally, if $M \subset L^{N}$ has codimension $3$ in $\R^{1,N}$ the normal holonomy is either vanishing or of type 2,
since $V$ is a lightlike $\nabla^{\perp}$-parallel vector field, i.e., $1\times SO(2)$, $SO(1,1)\times 1$ and
$SO_{0}(1,2)$ cannot be the normal holonomy group.
%
%
\bibliographystyle{amsalpha}
\nocite{*}
\bibliography{normalhol2-biblio}
\end{document}